\documentclass[12pt,fleqn]{scrartcl}

\usepackage[left=2cm,right=2cm,top=1.2cm,bottom=1.1cm,includeheadfoot]{geometry}
\usepackage[utf8x]{inputenc}
\usepackage[english]{babel}
\usepackage{amssymb,amsmath,amstext,amsthm,amsfonts}
\usepackage{pict2e}
\usepackage{bbm}

\usepackage{pdfpages}
\usepackage{graphicx}

\newtheorem{theo}{Theorem}[section]

\newtheorem{lem}[theo]{Lemma}
\newtheorem{prop}[theo]{Proposition}

\newtheorem{cor}[theo]{Corollary}

\newcommand{\la}{\langle}
\newcommand{\ra}{\rangle}

\newcommand{\R}{\mathbb{R}}
\newcommand{\N}{\mathbb{N}}
\newcommand{\eins}{{\mathbbm{1}}}

\renewcommand{\d}{\mathrm{d}}
\newcommand{\e}{\mathrm{e}}

\newcommand{\rZ}{{\cal Z}}
\newcommand{\rC}{{\cal C}}
\newcommand{\rP}{{\cal P}}
\newcommand{\rS}{{\cal S}}

\newcommand{\eyy}{= \\ &=&}

\newcommand{\pf}{ \longrightarrow }
\newcommand{\pfk}{ \rightarrow }

\newcommand{\D}{\mathrm{d}}
\newcommand{\oI}{{\textbf  I}}
\newcommand{\oA}{{\textbf  A}}

\newcommand{\oAs}{{\oA\!\!\!\;}^*}
\newcommand{\oTt}{{\textbf  T}(t)}
\newcommand{\oT}{{\textbf  T}}

\definecolor{rot}{rgb}{1.000,0.000,0.000}

\definecolor{gruen}{rgb}{0.000,1.000,0.000}

\usepackage{fancyhdr}
\begin{document}
\title{Memory equations as reduced Markov processes}
\author{Artur Stephan\thanks{Artur Stephan: Insitut f\"ur Mathematik, Humboldt-Universit\"at zu Berlin; Rudower Chaussee 25, D-12489 Berlin, Germany; email: stephan@math.hu-berlin.de}, Holger Stephan\thanks{Holger Stephan: Weierstra\ss{}-Insitut f\"ur Angewandte Analysis und Stochastik; Mohrenstra\ss{}e 39, D-10117 Berlin, Germany; email: holger.stephan@wias-berlin.de}}

\maketitle

\begin{abstract}
A large class of linear memory differential equations in one dimension,
where the evolution depends on the whole history, can be
equivalently described
as a projection of a Markov process living in a higher dimensional space. 
Starting with such a memory equation, we propose
an explicit construction of the corresponding Markov process.
From a physical point of view the Markov process can be understood as
a change of the type of some quasiparticles along one-way loops.
Typically, the arising Markov process
does not have the detailed balance property.
The method leads to a more realistic modeling of
memory equations. Moreover, it carries over the large number of investigation tools for
Markov processes to memory equations, like the calculation of the equilibrium
state, the asymptotic behavior and so on.
The method can be used for an approximative solution of
some degenerate memory equations like delay differential equations. 
\end{abstract}

~

%\tableofcontents

%%%%%%%%%%%%%%%%%%%%%%%%%%%%%%%%%%%%%%%%%%%%%%%%%%%%%%%%%%%%%%%%%%%

\newpage

\section{Introduction}

Memory equations describe the time evolution of some quantity,
considering the whole prehistory of the evolution:
The past influences the future.

Markov processes, or more generally time evolutions
with the Markov property, describe the problem under the assumption that the
evolution can be predicted, knowing only the current state: 
The present influences the future. 

At first glance, by means of memory equations, it is possible to
investigate a wider class of problems, since
evolution equations with the Markov property can be regarded as degenerate memory
problems, where the dependence of the past is concentrated in one moment.

But from a philosophical point of view, it seems to be natural that a
complete description of a problem has to be a Markov one for the following
reason: The Markov property means that the solution operator is a semigroup,
i.e. it is time-invariant. Due to Noether's theorem, this invariant
corresponds to the conservation of some energy, the dual variable of time.
Thus, the Markov property is the typical property of a model,
where some energy is conserved.

Conversely, if the evolution is governed by a non-Markovian equation, it is
not complete, some energy is lost. This requires finding more degrees of
freedom unless the model is Markovian.
In other words, it is to be expected that a non-Markovian description
can be regarded as some part or restriction of a more-dimensional Markov process.

This theoretical thought can be confirmed in various practical situations:
\begin{itemize}
\item An arbitrary (nonlinear) dynamical system on a compact space $\rZ$
can be
  equivalently formulated as a linear deterministic Markov process on the
  space of Radon measures on $\rZ$ (see, e.g. \cite{stephan2})
via its Liouville equation.
\item A general linear evolution equation that is nonlocal in space and time,
including
  jumps and memory on some
  domain in $\R^n$, can be understood as a limit of a diffusion process (a
  special Markov process) on a complicated Riemannian manifold 
(see \cite{stephan3}).
\item The projection of a general Brownian motion (a special Markov
process in
  phase space) on the coordinate space is a diffusion process if the initial
  velocity is Maxwellian (see \cite{stephan1}).
\end{itemize}
Hence, the idea that a memory equation can be regarded as part of a higher
dimensional Markov process, does not seem to be very surprising.
Indeed, the main result in this paper is that we provide the construction of an
   easily analyzable Markov process for a more or less arbitrary given
memory
   kernel.

Let us briefly revise the basic facts in modeling and analyzing Memory equations and Markov processes.

\subsection{Memory Equations}

Memory equations (ME) are differential equations where the evolution depends
not only on the current state but also on the past.  
MEs are a special case of functional differential equations - an equation of
unknown functions and their derivatives with different argument values. The
mathematical theory of functional differential equations (or integro-differential
equations) is treated in \cite{HaleLunell, FDE}.  

From the viewpoint of modeling and analysis, MEs have attracted a lot
of attention during the last decades. 
For example, they arise in modeling flows trough fissured media,
\cite{HornungShawalter, Peszynska} or in modeling heat conduction with finite wave speeds \cite{GurtinPipkin}. 
We consider MEs of convolution type. Such equations arise also as
effective limits of homogenization problems, starting with the pioneering work
of L. Tartar \cite{Tartar}. 

The object of interest is a linear memory equation of the form
\begin{align}\label{GeneralME}
\dot u (t) = &-a u + K \ast u =  -a u + \int_0^t K(t-s) u(s) \d s,~
u(0)=u_0, 
\end{align}
where $u:[0,\infty[\rightarrow \R$ is a scalar state variable, $u_0\in\R_{\geq
  0}$ and $K:\R_{\geq 0}\rightarrow \R_{\geq 0}$ is a positive real
kernel. Please note, we focus on a scalar variable, but our considerations can
be generalized to systems as well as to non-autonomous linear PDEs (like
diffusion equations with time-dependent diffusion coefficients). 

Let us briefly explain the ME (\ref{GeneralME}).
In contrast to $\dot u = -a u$, where the decay is quite fast, in this
equation the decay is damped due to the influence of former states. The ME can
be interpreted as a reduction of the mass into unknown depots. Phenomenologically,
this can be  modeled by $a = a(t)$, which yields a non-autonomous
equation. Another way to think about (\ref{GeneralME}) is the
following. Introducing the function $A$ defined by $A' = -K$ and $A(0)=a$, we
get 
\begin{align*}
\dot u (t)  = -A(0)u - \int_0^t A'(t-s)u(s)\d s =  -\frac{\d }{\d t}  \int_0^t A(t-s) u(s) \d s.
\end{align*}
Integrating the above equation, we get
\begin{align*}
 u(t) = u(0) - \int_0^t A(t-s)u(s)\d s
\end{align*}
that can be regarded as a continuous analogue of the time-discrete scheme
\begin{align}\label{e172}
 u_n = u_0 - a_1 u_{n-1} - a_2 u_{n-2} -\dots.
\end{align}
Equivalently, using partial integration we get
\begin{align*}
\dot u (t)  = -A(t)u_0 - \int_0^t A(t-s)\dot u(s)\d s.
\end{align*}
This form is often considered (e.g. in \cite{Peszynska}). Subsequently, we use the form (\ref{GeneralME}).

For solving a ME, the memory described by $K(t)$ or $A(t)$ has to be known for any time $t\geq0$. This is often postulated, i.e. $K(t)$ is given by heuristic arguments.

A typical and simple example is $K_\alpha(t) = \alpha \e^{-\alpha t}$ for $\alpha>0$. Then $K_\alpha(t)\geq0$ and $\int_0^\infty K_\alpha(t)\d t = 1$. 

In this case, for $\alpha \pf +\infty$, the integral on the right-hand side of (\ref{GeneralME}) tends to $u(t)$ -- the ME becomes an ordinary differential equation. 

In the same sense, a sequence of some other integrals of convolution type can
tend to a delay differential equation (DDE), that means $K(t) =
\sum_{j}\alpha_j \delta(t-t_j)$ for large enough $t\geq0$. So, the kernel $K$
can be interpreted as a measure on the time line that can be approximated by
the ``simplest'' measures: convex combinations of $\delta$-measures. Note
that DDEs with the above kernel of the form 
\begin{align*}
 \dot u = -a u + \sum_{j}\alpha_j u(t-t_j),
\end{align*}
are solved with respect to an initial condition $\phi\in\mathrm{C}([-\max\{t_j\},
0])$. That means the
solution space is infinite dimensional.
On the other hand regarding the modeling viewpoint, it is difficult to derive
an initial value $\phi\in\mathrm{C}([0,T])$ for a DDE. Often the initial value
$\phi$ is assumed to be constant or a simple given function.  See 
e.g. \cite{Smith} for more details, where the analysis and applications
especially for modeling aftereffect phenomena are presented. 

The ME needs the initial value only for one fixed value, say $t=0$. But, if $t\geq \max\{t_j\}$, the DDE become a ME. This means, that
the beginning of the evolution is also modeled in the ME. In this sense, MEs
include many types of differential equations like ODEs and DDEs. We remark
that also from the modeling viewpoint it is more natural to treat kernels that are not
located at precise time values but are smeared.

Another important property is the asymptotic behavior. The ME is a
non-autonomous differential equation. The equilibrium cannot be calculated
setting $\dot u = 0$. Assuming $\int_0^\infty K(t)\d t = a$, any constant
solution $u(t) = u_0$ satisfies  
\begin{align*}
 \lim_{t\rightarrow \infty} \left(-a u(t) + \int_0^tK(s)u(t-s)\d s\right) = 0.
\end{align*}
Assuming $\int_0^\infty K(t)\d t \neq a$, there is no non-trivial solution that
makes the right-hand side zero, so that it is no equilibrium of the ME.

\subsection{Markov Processes}

There is a huge amount of literature on Markov Processes (MP) -- see, e.g. 
\cite{bobrowski,durrett,dynkin}. 
Here we introduce our notation.

Let $\rZ$ be a given state space, a compact topological space,
$\rC := \rC(\rZ)$ the Banach space of continuous functions on $\rZ$ and
$\rP := \rP(\rZ)$ the set of probability measures, 
i.e. the subset of Radon measures
$p$ on $\rZ$ with $p \geq 0$ and $p(\rZ)=1$.

A family $\oT(t)$, $t \geq 0$ of linear bounded operators in  $\rC$
is called a \textit{Markov semigroup} if it is a semigroup, i.e. if it satisfies
\begin{eqnarray*}
\oT(t_1+t_2) = \oT(t_1)\oT(t_2),~\oT(0) = \oI,~t_1,t_2 \geq 0~,
\end{eqnarray*}
it is positive $\oT(t) \geq 0$ in the cone sense of $\rC$ and $\eins$, the
constant 
function is a fix-point of $\oT(t)$ for all $t \geq 0$, $\oT(t)\eins=\eins$. We refer to \cite{1184, Nagel}.
The semigroup property is often called \textit{Markov property} and it is
equivalent to the 
assumption that the trajectory depends only on the present time point
and not on the past.

A linear operator $\oA$ on $\rC$ is called  \textit{Markov generator} if it is
the generator of a Markov semigroup, i.e.  if $g(t) = \oT(t) g_0$, where
$\oT(t)$ is a Markov semigroup. Then $g(t) = \oT(t) g_0$ is the solution of the equation
\begin{eqnarray}\label{e422}
\dot{g}(t) = \oA g(t),~g(0)=g_0
\end{eqnarray}
for an initial value $g_0$ from the domain of $\oA$.
This equation is called \textit{backward Chapman-Kolmogorov equation}.
A MP is the result of the action of the adjoint semigroup
$\oT^*(t)$ at a probability measure $p_0$, i.e. $p(t)=\oT^*(t)p_0$.
Any MP has at least one stationary probability measure
$\mu\in\rP$. It satisfies
$\oT^*(t) \mu = \mu$ for all $t \geq 0$. This is a consequence of the
Markov-Kakutani Theorem. The stationary probability measure $\mu$ is an element of the null-space of $\oAs$.

In this paper we consider 
continuous-time MPs on discrete state spaces.
$\rZ=\{z_0,...,z_N\}$ is a finite set of $N+1$ states.
In this case, we have $\rC=\R^{N+1}$ and $\rP$ is the simplex of 
probability vectors $\rP:=\textrm{Prob}(\{z_0,\dots,
z_N\}):=\{p\in \R^{N+1}: p_i\geq0, \sum_{i=0}^{N+1} p_i =1\}$
and a subset of $\R^{N+1}$, 
too.
A Markov semigroup is a real
matrix family $\oTt$ on $\R^{N+1}$ with positive entries
and row sum 1. Its adjoint is the transposed matrix family $\oT^*(t)$.

A MP is $p(t) = \oT^*(t) p_0$, where $p_0$ is some given 
probability vector. It satisfies the set of equations
\begin{eqnarray}\label{e421}
\dot p (t) = \oAs p(t),~p(0)=p_0,
\end{eqnarray}
where $\oAs$ is the adjoint of the corresponding Markov generator.
This equation is called \textit{forward Chapman-Kolmogorov equation}. In contrast to
equation (\ref{e422}) describing the evolution of  moment functions,
equation (\ref{e421}) describes the evolution of probability vectors. This
means that one component of the vector $p (t)$ can be understood as the
probability of the corresponding state, regardless of the 
probability of the other states.

It is well known that equation (\ref{e421}) has a unique solution $p(t) \in
\rP$ if and only if the off-diagonal
elements are nonnegative and the columns of $\oAs$ sum up to zero.
Thus, for $\oA = (A_{ij})$ we have $A_{ij} \geq 0$ for $i\not= j$ and 
$A_{ii} = - \sum_{i\not=j=1}^n A_{ij}$.

For a generic 
Markov matrix the stationary probability $\mu$ is unique and all trajectories
$\oT^*(t) p_0$ for 
any initial state $p_0$ converge to $\mu$. 
We only consider MPs with a unique stationary probability.

The eigenvalues of a Markov
generator have always strongly negative
real part, except one eigenvalue 0. The corresponding eigenvector is
$\eins$ for $\oA$ and $\mu$ for $\oAs$. If the eigenvalues
$\lambda_i$
of $\oAs$ are all different, every component of
the solution to (\ref{e421}), i.e. every component of $\oT^*(t) p_0$ is a 
linear combination of $\eins$ and exponential decaying functions
$\e^{-\lambda_i t}$.

A MP in $\R^{N+1}$ allows for different physical
interpretations. Apart from the 
canonical interpretations as a probability vector, it can be understood as
some concentration or amount of $N+1$ different materials. We will follow this
interpretation 
and will assume that this amount of materials is represented by 
particles of different types. These particles can transform into each other,
changing their type, which can be understood as a linear reaction. 
The entries of the Markov matrix $A_{ij}$ describe the rates of transforming
particles of type $z_j$ into particles of type $z_i$. 
Therefore, if we are only interested in the amount of material of one type, it
is enough to consider the corresponding component of the vector $p (t)$
only. 
The initial amount
of material is $p_0$.
Since $\oA$ is a Markov generator, positivity of the concentration and
the whole mass is conserved.

If a Markov generator $\oA = (A_{ij})$ and its stationary state
$\mu = (\mu_i)$ satisfy $A_{ij}\mu_j = A_{ji}\mu_i$ for any $i,j\in\{1,\dots,
n\}$, it is said 
that the corresponding MP has the detailed balance property.
It is equivalent to the case that the matrix $ (A_{ij})$ is symmetric
in the $L^2$-Hilbert space over $\mu$. Such a matrix has to have real
eigenvalues. We remark that the opposite is not true in general: A Markov
process without the detailed balance can have real eigenvalues, too. Moreover,
there can be no Hilbert space at all, where it is symmetric. 
From a physical point of view, the condition  $A_{ij}\mu_j = A_{ji}\mu_i$ 
means that any transition $z_i \Leftrightarrow z_j$ is in a local
equilibrium.
Thus, the detailed balance case is easier to analyze but it rarely appears in
general. The systems that we consider do not have the 
detailed balance property in principle.

\subsection{What our paper deals with}

In this paper, we connect the concepts of Markovian dynamics and 
non-Markovian dynamics, which seem to be different at first glance. 
Starting with a MP of a special form, we conclude a ME for the first 
coordinate. The ME is a scalar differential equation, but our considerations can also be applied to PDEs.
The resulting MP can be physically understood; the ME is governed 
by a kernel which is the sum of exponential functions. Then another path 
is taken: Starting with a ME with an exponential kernel, we find a MP where
its first component again yields the ME.
The other components can be understood as hidden degrees of freedom that
have to be included in a complete description of the problem.
This procedure is not unique and 
thus, it cannot be said that the hidden degrees of freedom are 
real physical variables. On the other hand, the construction of the 
MP out of the kernel is intuitive since the kernel is approximated by 
its moments. This method can be used to approximate a general positive 
kernel taking the enlargement of the MP into account. The simple case of 
two and three states is presented in chapter \ref{SimpleExample}. In 
this case, all solutions and kernels can be calculated by hand.  In 
chapter \ref{GeneralTheory} we consider the general case. The main 
theorems are stated here.

The method has many physical and mathematical advantages -- both for the theory of MPs and MEs. 
We want to highlight only two of them.
Firstly, the modeling of a kernel for ME is usually done by heuristic 
arguments. The method presented here can be used to model kernels in a 
more convenient manner, since the MP has an underlying physical meaning. 
Moreover, the modeling of the beginning of the process is also done. Secondly, the asymptotic behavior of a non-autonomous 
differential equation can immediately be calculated from the Markovian 
dynamics.

The paper concludes with chapter \ref{DDE}. Here we note the connection to delay
differential equations, where the kernel is highly degenerate. This is also
reflected in the setting of MP: The underlying Markov generator has a very
special form. We observe that the solution of the ME converges to the
equilibrium of the MP. The spectral functions of ME and MP also converge.

Summarizing, we have the following connection of modeling levels:
\begin{center}
 MP $\subset$ DDE $\subset$ ME $\subset$ MP'.
\end{center}
Here MP' is a Markov process with a larger number of degrees of freedom.

It is well known that a linear delay equation with delay $T$ in a state space $X$ can be regarded as an autonomous equation in a much
larger space $\mathrm C([-T,0], X)$, see e.g. \cite{Nagel}. There, the evolution of the delay
equation
is described by a semigroup of linear operators.
This approach is not our aim in this paper. In our setting, the space of the MP' is not so large typically.

\textbf {Notion:} In this paper, the Laplace transform is frequently used. Some properties are 
summarized in the appendix. MEs of convolution
type have the important property that the Laplace transform maps them into
multiplication operators. The Laplace transform $\mathcal{L}(u)$ of a real valued
function $t\mapsto u(t)$ is defined by $\mathcal{L}(u)(\lambda)=\hat u(\lambda) =
\int_0^\infty\e^{-\lambda t} u(t)\d t$. If there is no confusion, we omit the 'hat' on $\hat
u$ and just write $u$ or $u(\lambda)$.

Some analytical tools concerning Lagrange polynomials
and simplex integrals are presented in the appendix, too.

%\newpage
%~\newpage

\section{Some simple Markov processes and memory equations}\label{SimpleExample}

Before starting the general theory, we firstly present the basic ideas 
focusing on simple low dimensional examples -- MPs with two and 
three states. Apart from the sake of simplicity nearly all phenomena 
of the general theory are eminent. 

\subsection{Two states}
We consider a MP on a state space of two abstract states
$\{z_0, z_1\}$, generated by the Markov generator
\begin{align}\label{MP2Dims}
\oA=\begin{pmatrix}
 -a & a \\
  b & -b
\end{pmatrix} , \mathrm{~~and~its~transpose}~~
\oAs=\begin{pmatrix}
 -a & b \\
  a & -b
\end{pmatrix}. 
\end{align}

\begin{minipage}[c]{12cm}
The matrix $\oAs$ describes the switching between the two states with given rates
$a \geq 0$, $b \geq 0$.
We can think of an amount of matter, represented by particles,
which can occur in two types. For some reason we are interested
only in particles of 
the first type.
\end{minipage}
\hfill
\begin{minipage}[t]{4.5cm}
\unitlength=2cm
\begin{picture}(2.5, 0.8)
\linethickness{0.2mm}
\put(0.25,0.25){\circle*{0.1}}         
% 0 -> 0
\put(1.75,0.25){\circle*{0.1}}         
% 1 -> 1
%
\put(0.5,0.2){\vector(1,0){1.0}}      
% 0 -> 1
\put(1.5,0.3){\vector(-1,0){1.0}}       
% 1 -> 0
\put(0.05,0.17){\makebox(0.0,0.0){$z_0$}}
\put(1.95,0.15){\makebox(0.0,0.0){$z_1$}}

\put(1,0.41){\makebox(0.0,0.0){$b$}}    % 3 -> 2
\put(1,0.08){\makebox(0.0,0.0){$a$}}    
% 2 -> 3
\end{picture}
\end{minipage}
\vspace*{0.2em}

The equation describing the evolution of the vector $p=(u,v)$ 
reads $\dot p = \oAs p$ with $p(0)=p_0$. We assume 
that in the beginning the total mass is concentrated in the first 
variable, i.e. $p_0=(u_0, 0)$. In other words, all particles have type $z_0$.

The eigenvalues of $\oAs$ are $\{0, -(a+b)\}$.
The stationary solution is 
$\mu = \left(\frac{b}{a+b} u_0, \frac{a}{a+b} u_0 \right)$. It is unique
unless the non interesting case $a=b=0$. Any MP with two
states has the detailed balance property.

For $(u,v)$ the system reads as
\begin{equation}\label{e601}
 \left\{
 \begin{aligned}
  \dot u &= - a u + b v\\
 \dot v &= a u - bv.
 \end{aligned} \right.
\end{equation}

Using the Laplace transform and writing $u(\lambda) = \mathcal{L}(u(t))(\lambda)$ and $v(\lambda) = \mathcal{L}(v(t))(\lambda)$, we obtain a system of equations for $(u,v)$ in the form
\begin{equation*}
 \left\{
 \begin{aligned}
  (\lambda + a) u - u_0&= b v\\
 (\lambda + b) v &= a u.
 \end{aligned} \right.
\end{equation*}
This yields an equation for $u$ in the form 
\begin{align*}
(\lambda + a) u - u_0 = \frac{b a}{\lambda + b}u \Rightarrow \lambda u - u_0 = - a u + \frac{b a}{\lambda + b}u. 
\end{align*}
Using the inverse Laplace transform, we obtain a Memory Equation for $u$
\begin{align}\label{e602}
\dot u = -a u + a b \int_0^t \e^{-b(t-s)}u(s)\d s = -a\frac{\d}{\d t} \int_0^t \e^{-b(t-s)}u(s)\d s.
\end{align}
%\vspace{-1em}
\begin{minipage}[b]{8cm}
The kernel $K(t)=b \e^{-bt}$ describes a dependence of the current state from
previous time moments. For $b \pf \infty$, $K(t)$ tends to $\delta(t)$ and the
equation becomes $\dot u = 0$. 
Thus, the right hand side of equation (\ref{e602}) consists of two terms,
the first one, $-a u$ describe an exponential decay, whereas the second one,
the
memory term describe an opposite effect: Particles that disappear, 
occur after a while. 
\end{minipage}
\hfill
\begin{minipage}[b]{7.0cm}
\unitlength=1cm
\begin{picture}(6,5)
\linethickness{0.2mm}
\put(0.0,0.0){\includegraphics[width=6cm]{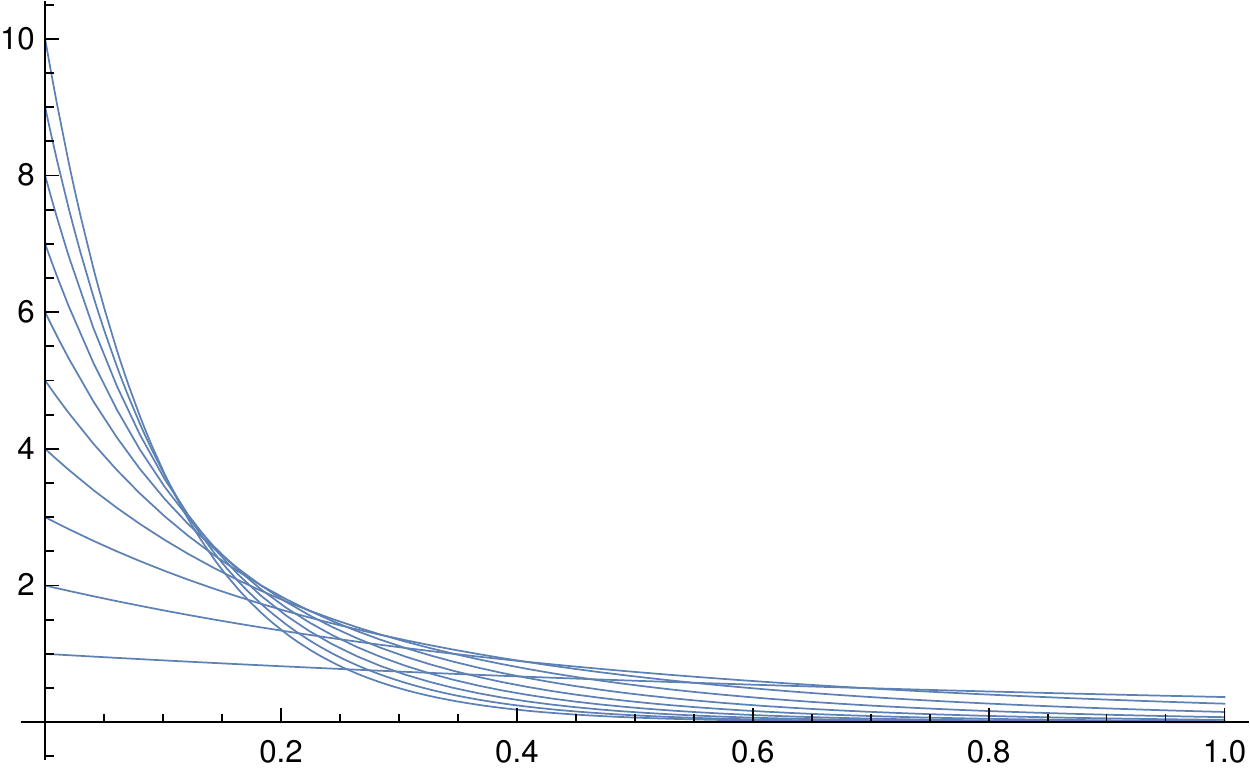}
}
\put(0.7,2.3){Kernel $b\e^{-bt}$ for $b=1, \dots, 10$}
\put(-0.5,3.2){\makebox(0.0,0.0){$K(t)$}}
\put(6.5,0.1){\makebox(0.0,0.0){$t$}}
\end{picture}
\end{minipage}%\newpage

The time that passes between disappearing and
reappearing, decreases with $1/b$. 
In the end, not all matter disappears like in a pure equation
$\dot u = -a u$ but an equilibrium between 
disappearance and reappearance arises.

The same effect is caused by the MP, changing the type of the
particles. The particle changes the type from $z_0$ to $z_1$ with rate $a\geq0$,
it seems to disappear, if we look only at type $z_0$. After a while it re-changes to
type $z_1$ (it occurs) with rate $b\geq0$. This give the 
exponential time behavior $\e^{-bt}$ (corresponding to the memory
kernel $K(t)=b \e^{-bt}$), characteristic for MPs.

The equation (\ref{e602})
-- or equivalently the system (\ref{e601}) -- can be solved
explicitly. We obtain for the Laplace transform
\begin{eqnarray*}
u(\lambda) = {\lambda + b \over \lambda (\lambda  + a + b) } u_0 =
\left( {b \over a + b } {1 \over \lambda } + {a \over a + b }
 {1 \over \lambda +a+b} \right) u_0 
\end{eqnarray*}
and for the solution itself
\begin{eqnarray*}
u(t) = {b \over a + b } u_0 +  {a \over a + b } e^{-(a+b)t} u_0
\end{eqnarray*}
The solution tends to an equilibrium state $u_\infty = {b \over a + b } u_0$,
the first component of the stationary solution $\mu$.

It is not possible to calculate it from the memory equation (\ref{e602}), 
directly. Setting $\dot u =0$,  the equation 
\begin{eqnarray*}
\dot u = -a u + a b \int_0^t \e^{-b(t-s)}u(s)\d s = -a\frac{\d}{\d t} \int_0^t \e^{-b(t-s)}u(s)\d s.
\end{eqnarray*}
does not have any solution at all. Passing to the limit $t \pf \infty$ (and rewriting
at first
$ \int_0^t \e^{-b(t-s)}u(s)\d s =  \int_0^t \e^{-bs}u(t-s)\d s$)
we obtain
\begin{eqnarray*}
0 = -a u_\infty + a b \int_0^\infty \e^{-bs}u_\infty \d s~.
\end{eqnarray*}
Any constant $u_\infty$ solves this equation. This strange behavior of the solution of memory equations is typical
and can be illustrated in
a picture, showing the time behavior of both,
the solution of the MP and their first component -- the solution of the memory equation.
\newpage
\begin{minipage}[b]{7cm}
Investigating only the solution of the memory equation, it is not clear
why the \textcolor{green}{trajectory $u(t)$} stops in $u_\infty$.
Whereas looking from above, the \textcolor{rot}{trajectory $(u(t),v(t))$}
has to stop at the stationary state $\mu$, the intersection of the subspace
$u+v=1$ with the null space of $\oAs$.
\end{minipage}
\hfill
\begin{minipage}[b]{8cm}
\unitlength=1cm
\begin{picture}(5,5)
\linethickness{0.2mm}
\put(-0.3,0.0){\vector(1,0){6.3}}   
\put(0.0,-0.3){\vector(0,1){4.8}}   
\put(0,0){\line(1,2){2}}   
\put(0,4){\line(5,-4){5}}   
\put(1.,2.8){\makebox(0.0,0.0){$\mu$}} 
\put(5.5,-0.2){\makebox(0.0,0.0){$u$}} 
\put(-0.3,4){\makebox(0.0,0.0){$v$}} 
\put(5.,-0.3){\makebox(0.0,0.0){$u_0$}} 
\put(1.4,-0.3){\makebox(0.0,0.0){$u_\infty$}} 
\put(5,0){\circle*{0.2}}
\put(1.4,0){\circle*{0.2}}
\put(5.0,0.1){\textcolor{rot}{\vector(-5,4){3.5}}}
\put(5.0,0.1){\textcolor{green}{\vector(-1,0){3.5}}}
\end{picture}
\end{minipage}

%\newpage

\subsection{Three states}

A general memory kernel has not be concentrated in $t=0$. It can describe a
transfer of mass from a very earlier time. It seems that this situation can be modeled
by transitions between many quasiparticles before it appears at its starting type
again.

\begin{minipage}[b]{11cm}
To understand the action of such a transition loop, we investigate in
detail a special case of three states, namely the transformation
of a fixed particle (type $z_0$) in
two different quasiparticles. One of them (type $z_1$) can be transformed back
into type $z_0$ immediately, whereas the other (type $z_2$)  can be transformed back
into type $z_0$ only by two steps, changing at first to type $z_1$. This process is
illustrated in the picture.
\end{minipage}
\hfill
\begin{minipage}[b]{5cm}
\unitlength=2cm
\begin{picture}(2.5,1.8)
\linethickness{0.2mm}
\put(0.25,0.25){\circle*{0.1}}         % 1 -> 1
\put(1.75,0.25){\circle*{0.1}}         % 2 -> 2
\put(1.00,1.25){\circle*{0.1}}         % 3 -> 3
%
%\put(0.45,0.2){\vector(1,0){1.1}}      % 2 -> 3
\put(1.5,0.3){\vector(-1,0){1.}}       % 3 -> 2
%\put(1.7,0.45){\vector(-3,4){0.55}}    % 3 -> 1
\put(1.1,1.05){\vector(3,-4){0.5}}     % 1 -> 3
\put(0.41,0.40){\vector(3,4){0.5}}     % 2 -> 1
\put(0.85,1.18){\vector(-3,-4){0.55}}  % 1 -> 2
\put(0.05,0.17){\makebox(0.0,0.0){$z_1$}}
\put(1,1.45){\makebox(0.0,0.0){$z_0$}}
\put(1.95,0.15){\makebox(0.0,0.0){$z_2$}}
\put(0.45,0.9){\makebox(0.0,0.0){$a_1$}}  % 1 -> 2
%\put(1.6,0.9){\makebox(0.0,0.0){$c$}}   % 3 -> 1
\put(0.8,0.65){\makebox(0.0,0.0){$b_1$}}  % 2 -> 1
\put(1.25,0.65){\makebox(0.0,0.0){$a_2$}} % 1 -> 3
\put(1,0.41){\makebox(0.0,0.0){$b_2$}}    % 3 -> 2
%\put(1,0.08){\makebox(0.0,0.0){$b$}}    % 2 -> 3
\end{picture}
\end{minipage}

\subsubsection{From Markov to Memory}
The simple MP on a state space of three abstract states $\{z_0,z_1,z_2\}$ is described by the Markov generator
\begin{align}\label{MP3Dims}
\oA=\begin{pmatrix}
 -a_1-a_2 & a_1 & a_2\\
  b_1 & -b_1 & 0\\
  0 & b_2 & -b_2
\end{pmatrix} ,~~
\oAs=\begin{pmatrix}
 -a_1 - a_2 & b_1 & 0\\
  a_1 & -b_1 & b_2\\
  a_2 & 0 & -b_2
\end{pmatrix} 
\end{align}
with $a_1,a_2,b_1,b_2 \geq 0$. 
The equation, generating the 
MP is 
\begin{align}\label{MarkovSystem3Reservoirs}
 \dot p (t) = \oAs p(t), ~~p(0)=p_0. 
\end{align}
Note, this is a Markov generator depending on four rates. A general Markov
generator on $\R^3$ depends on six rates.

The stationary state $\mu$ is the solution to $\oAs\mu = 0$ and can be
calculated easily as 
\begin{eqnarray*}
\mu =
\left(1+ \frac{a_1+a_2}{b_1}+ \frac{a_2}{b_2}\right)^{-1}
\left(1, \frac{a_1+a_2}{b_1}, \frac{a_2}{b_2}\right)u_0
=
{(b_1 b_2 , a_1 b_2 + a_2 b_2 , a_2 b_1 ) \over 
b_1 b_2 + a_1 b_2 + a_2 b_2 + a_2 b_1  } u_0.
\end{eqnarray*}

The eigenvalues (they have always non-positive real part) of the matrix are 
$\lambda_0 = 0$ and
\begin{eqnarray*}
\lambda_{1,2}=
-\frac 1 2 \left(a_1+a_2+b_1+b_2 \pm 
\sqrt{(a_1+a_2+b_1+b_2)^2-4
   (a_1 b_2+a_2 b_1+a_2 b_2+b_1
   b_2)}\right).
\end{eqnarray*}
Depending on $a_1,a_2,b_1,b_2$ the eigenvalues can be real 
(e.g. $\lambda_1 = -5$, $\lambda_2 = -11$ for $a_1=2,a_2=5,b_1=8,b_2=1$) or complex (e.g. for 
$\lambda_{1,2} = -9 \pm 2i$ for $a_1=2,a_2=5,b_1=8,b_2=3$).
(By the way, these are suitable values for an explicite solution with rational terms,
only.)

This MP has the detailed balance property, if $b_1 b_2 a_2=0$,
which is not interesting, since the coupling chain is broken.
Roughly speaking, the detailed balance property means that for any loop
in one direction there is a loop backwards with the same product of the rates.
But this is not the case in our model.
Thus, the MP under consideration violate 
the detailed balance property, generically.

The stationary state is unique if and only if the real parts of $\lambda_{1,2}$ are
strongly negative. Or, equivalently,
$b_1 b_2 + a_1 b_2 + a_2 b_2 + a_2 b_1 =0$. Since the $a_i,b_i$ are non negative,
this is a non interesting case that we exclude.
Then, the  stationary state is the equilibrium state for any
initial value.
Note, that nevertheless some of the $a_i,b_i$ might be zero.

As in the case of two states, we are interested only in the state $z_0$ of the system and ask for an evolution equation of this state. To do this, we introduce the notion 
$p=(u, v_1, v_2)$ and look for the evolution of $u$ with
an initial state $p_0=(u_0, 0, 0)$. This is naturally, since the states $z_1$ and
$z_2$ are unknown, and there is no reason to assume that particles with $z_1$, $z_2$ exist in the beginning. 

Equation (\ref{MarkovSystem3Reservoirs}) is now equivalent to the system
\begin{equation*}
 \left\{
\begin{array}{rcrrr}
\dot{u}(t)   &=& - (a_1 + a_2) u(t) & + b_1 v_1(t) &  \\
\dot{v}_1(t) &=& a_1 u(t) &- b_1 v_1(t)  &+ b_2 v_2(t)  \\
\dot{v}_2(t) &=& a_2 u(t)&   &- b_2 v_2(t)
\end{array}
\right. .
\end{equation*}
Passing to the Laplace transform, we obtain with  
$u = \mathcal L u$, $v_i = \mathcal L v_i$ the system
\begin{equation*}
\left\{
\begin{array}{rcrrrr}
\lambda{u}   &=& - (a_1 + a_2) u & + b_1 v_1 &  &+ u_0\\
\lambda {v}_1 &=& a_1 u &- b_1 v_1  &+ b_2 v_2 & \\
\lambda v_2 &=& a_2 u&   &- b_2 v_2&
\end{array}
\right. .
\end{equation*}
or equivalently, introducing $a=a_1+ a_2$, we get
\begin{equation*}
\left\{
\begin{aligned}
(\lambda + a) u - u_0&=&  b_1 v_1 \\
(\lambda + b_1) v_1 &=& a_1 u + b_2 v_2\\
(\lambda + b_2) v_2 &=& a_2 u
\end{aligned}
\right. .
\end{equation*}
Here, $v_1$ and $v_2$ can be eliminated as
\begin{align*}
v_2 = {a_2 \over \lambda+b_2} u~,~~
v_1 = 
{a_1 \over \lambda+b_1} u + {b_2 \over \lambda+b_1} v_2
= 
{a_1 \over \lambda+b_1} u + {a_2 b_2 \over (\lambda+b_1)(\lambda+b_2) } u .
\end{align*}
We conclude the following equation for $u$
\begin{align} \label{Markov3Reservoir1Equation}
 \lambda u -u_0 =\left(-a  + a_1\frac{b_1}{\lambda+b_1} + a_2\frac{b_1}{\lambda+b_1} \frac{b_2}{\lambda+b_2}\right)u. 
\end{align}
This is an equation for the first state, only. 
It can be solved explicitly with respect to $u$.
But, at this moment, this is not our aim. We are looking for an equation 
for $u$.
We write
\begin{align*}
\frac{b_1}{\lambda+b_1} \frac{b_2}{\lambda+b_2} = \frac{b_1b_2}{b_2-b_1}\left( \frac{1}{\lambda+b_1}- \frac{1}{\lambda+b_2}\right),
\end{align*}
and, after transforming inverse, we get an equation
for the function $u(t)$, namely
\begin{align}\label{ME3Reservoirs}
\dot u 
&= -a u + a_1b_1 \int_0^t \e^{-b_1s}u(t-s)\d s + a_2\frac{b_1b_2}{b_2 - b_1} \int_0^t (\e^{-b_1s}-\e^{-b_2s})u(t-s)\d s\\
&= -a u + (K\ast u) (t)\nonumber, 
\end{align}
where 
\begin{align}\label{KernelME3}
K(t)&=
b_1 a_1 \e^{-b_1t}+a_2\frac{b_1b_2}{b_2-b_1}\left(\e^{-b_1t}-\e^{-b_2t}\right)
=\\
&= \label{e612}
\left(b_1 a_1 + \frac{b_1b_2a_2}{b_2-b_1}\right)
\e^{-b_1t}-\frac{b_1b_2a_2}{b_2-b_1}\e^{-b_2t} .
\end{align}
So, we obtain a memory equation with the kernel $K$.
This equation describe the evolution of the first state of our
physical system, depending on the whole past from $0$ to time $t$.
Obviously, this dependence is a result of the projection, 
since nothing else had be done. 
Thus, $u(t)$ is the solution of two equivalent equations, a 
memory equation and a component of a Markov system.

~

\begin{minipage}[b]{8cm}
The kernel $K(t)=a_1 K_1(t)+a_2 K_2(t)$ is the sum of two parts
\begin{eqnarray*}
K_1(t) &=& 
b_1 \e^{-b_1t}\\
K_2(t) &=& \frac{b_1b_2}{b_2-b_1}\left(\e^{-b_1t}-\e^{-b_2t}\right)
\end{eqnarray*}
each of them is obviously positive . 
If we denote $m_i=\int_0^\infty t K_i(t) \d t$ the mean
time of a kernel, we have
\begin{eqnarray*}
m_1 = {1 \over b_1},~
m_2 = {1 \over b_1}+{1 \over b_2}.
\end{eqnarray*}
\end{minipage}
\hfill
\begin{minipage}[b]{8cm}
\unitlength=1cm
\begin{picture}(6,5)
\linethickness{0.2mm}
\put(0.0,0.0){\includegraphics[width=7cm]
{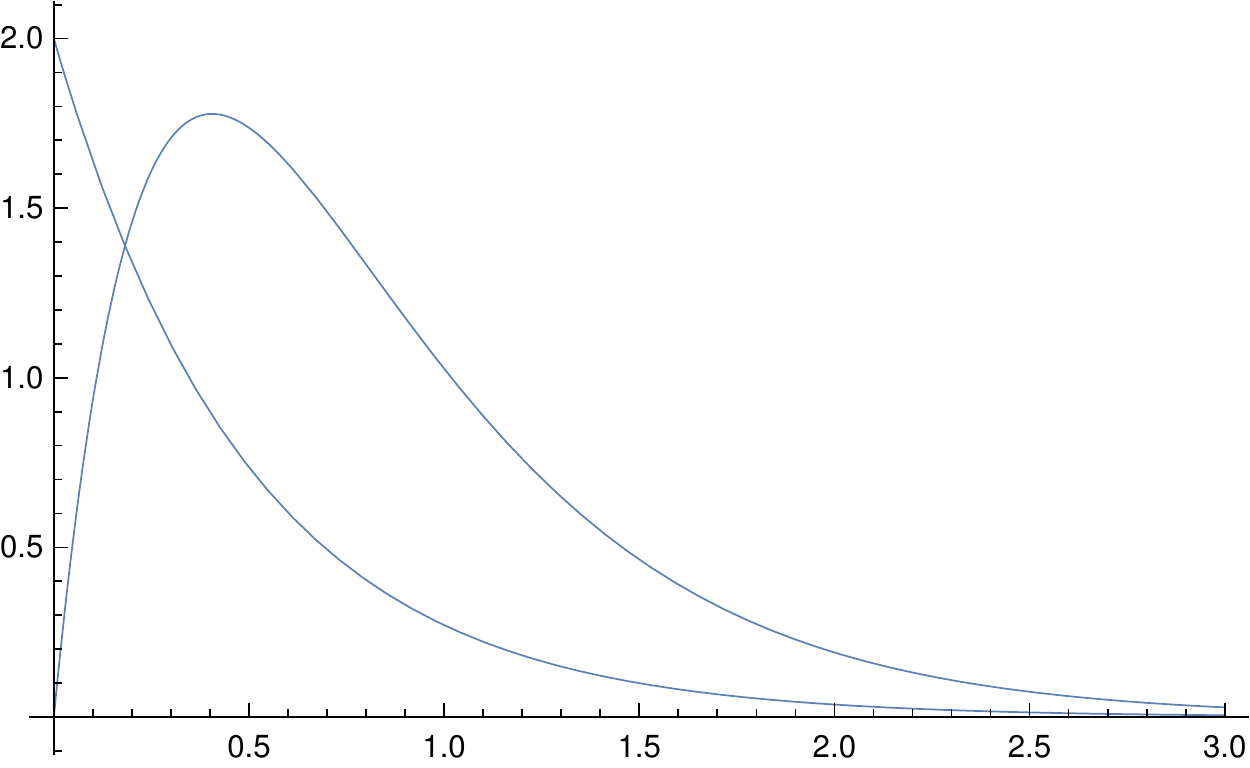}}
\put(1.0, 1.0){$K_1$}
\put(3.0, 2.0){$2K_2$}
\put(3.0, 3.5){$b_1=2$, $b_2=3$}
\put(-0.5,3.6){\makebox(0.0,0.0){$K(t)$}}
\put(6.7, 0.6){\makebox(0.0,0.0){$t$}}
\end{picture}
\end{minipage}

~

The first kernel $K_1$ describes a memory effect with 
small mean time and correspond to a small loop 
$z_0 \stackrel{a_1}{\pf} z_1 \stackrel{b_1}{\pf} z_0$
in the MP. The other kernel $K_1$ describes a memory effect with 
longer mean time and correspond to a longer loop 
$z_0 \stackrel{a_2}{\pf} z_2 \stackrel{b_2}{\pf} z_1 \stackrel{b_1}{\pf} z_0$. The relative coefficients $a_i/a$ form a convex combination. The
transitions $z_0 \stackrel{a_i}{\pf} z_i$ split the whole number of
particles in parts according to the loops.

Let us summarize some properties of the kernel $K(t)$.
\begin{itemize}
\item $K(t)$ is the sum of exponential decaying functions, where the exponents
  are diagonal elements of $\oA$.
\item The arising memory equation is (\ref{ME3Reservoirs}) with
$a = \sum_i^{N} a_i$ or, equivalently,  $k(\lambda = 0)=a$
\item $K(t) \geq 0$ iff $k(\lambda) \geq 0$, since $a_i,b_i \geq 0$.
\end{itemize}

Equation (\ref{Markov3Reservoir1Equation}) can be solved
explicitely:
\begin{eqnarray*}
u \left( \lambda + a - {a_1 b_1\over \lambda+b_1} -
 {a_2 b_1 b_2 \over (\lambda+b_1)(\lambda+b_2) } \right) &=& u_0 \\
\Rightarrow \lambda u \left( {\lambda ^2 + \lambda (a+b_1+b_2) + 
a_2 b_1 + a_1 b_2 + a_2 b_2 + b_1 b_2 
\over (\lambda+b_1)(\lambda+b_2) } \right) &=& u_0 \\
\Rightarrow u = { 1 \over \lambda}
{ (\lambda+b_1)(\lambda+b_2) \over \lambda ^2 + \lambda (a+b_1+b_2) + 
a_2 b_1 + a_1 b_2 + a_2 b_2 + b_1 b_2 } && u_0.
\end{eqnarray*}
To get an explicite term for $u(t)$ we have to factorize the
denominator what leads -- of course -- to the same time behavior
as determined by the eigenvalues for the MP.

We compute the asymptotic behavior of the solution $u(t)$,
using the asymptotic properties of the Laplace transform.
We obtain for the equilibrium state
\begin{eqnarray*}
u_\infty = \lim\limits_{\lambda \pfk 0} \lambda u =
{  b_1 b_2 \over 
a_2 b_1 + a_1 b_2 + a_2 b_2 + b_1 b_2 } u_0.
\end{eqnarray*}
For the other components we get in the same manner
\begin{eqnarray*}
v_1(t=\infty) &=& 
{a_1 b_2 + a_2 b_2  \over 
a_2 b_1 + a_1 b_2 + a_2 b_2 + b_1 b_2 } u_0,\\
v_2(t=\infty) &=& 
{  a_2 b_1 \over 
a_2 b_1 + a_1 b_2 + a_2 b_2 + b_1 b_2 } u_0.
\end{eqnarray*}
These are the parts of the initial mass that remain in the states $z_1$ and $z_2$.

\subsubsection{From Memory to Markov}

Now, we go the opposite direction and start with a kernel that is
the sum of two exponential decaying terms, i.e.
\begin{eqnarray}\label{e661}
K(t) = c_1 \e^{-\alpha_1 t}  + c_2 \e^{-\alpha_2 t} 
\end{eqnarray}
with some real coefficients $c_1,c_2$. We assume $c_i\not=0$, otherwise
we are in the case of 2 states.
For definiteness, we assume $\alpha_1 > \alpha_2 > 0$. 
The $\alpha_i$ has to be strongly positive, otherwise we have no 
decreasing of the time dependence of the past.

This kernel has to be written in the form (\ref{KernelME3}) with positive
coefficients. We have
\begin{eqnarray*}
 K(t) &=& c_1 \e^{-\alpha_1 t}  + c_2 \e^{-\alpha_2 t} 
\eyy
(c_1+c_2) \e^{-\alpha_1 t}  + { c_2 (\alpha_1 - \alpha_2)}
{ \e^{-\alpha_2 t}  - \e^{-\alpha_1 t} \over  \alpha_1 - \alpha_2}.
\end{eqnarray*}
Thus, we have to demand $c_1 + c_2 \geq 0$ and $c_2 \geq 0$.
Both are consequences of the positivity of $K(t)$, setting $t=0$ and
$t \pf \infty$.

Now, the MP is easily constructed. We set
\begin{eqnarray*}
b_1 &=& \alpha_1 \\
b_2 &=& \alpha_2 \\
a_2 &=& {c_2(\alpha_1 - \alpha_2) \over \alpha_1 \alpha_2 }  \\
a_1 &=& {c_1+c_2 \over \alpha_1}~.
\end{eqnarray*}
The entries of the matrix $b_1,b_2,a_2$ are strongly positive, $a_1$
is non negative. This guarantees the uniqueness of the stationary solution.
Moreover, it violates the detailed balance property.

The existence of a positive equilibrium is fulfilled, we have
the
equation
\begin{eqnarray*}
\dot{u} = - a u + \int_0^t K(t-s) u(s) \d s,~u(0)=u_0,
\end{eqnarray*}
with the property of consistency $k(0)=a= a_1+a_2 = \frac{\alpha_1 c_2 + \alpha_2 c_1}{\alpha_1\alpha_2}$. Summarizing, we get the following result:

\begin{prop}~
The first component of the MP generated by $\oAs$ given by (\ref{MP3Dims})
is the solution to the ME (\ref{ME3Reservoirs}).

For a ME $\dot u= -a u + (K\ast u)$ with a kernel (\ref{e661}) 
with parameters $c_1,c_2,\alpha_1,\alpha_2$ satisfying
$\alpha_1 > \alpha_2 > 0$, $c_1 + c_2 \geq 0$ and $c_2 \geq 0$, it can be
constructed a three dimensional MP, where the first component 
coincides with the solution to the ME.
\end{prop}

%\newpage

\section{General Memory Equations as Markov processes}\label{GeneralTheory}

In this chapter, we generalize the ideas from the last chapter to an 
arbitrary finite dimensional MP.
Firstly, we show that the first coordinate of a special MP,
consisting of different transformation loops, satisfies 
a suitable memory equation with a more or less general kernel. 
Then, we go the opposite direction: We show that a ME with a kernel 
of a special form yields the MP we started with. The construction of the MP
is explicitly.

\subsection{From Markov to Memory}

We consider a MP of $N+1$ abstract states 
$\{z_0, z_1, \dots, z_N\}$ of the following form

\begin{align}\label{MarkovGeneratorGeneralCase}
\oAs=\begin{pmatrix}
 -a & b_1 & 0 & 0 & \dots &0\\
  a_1 & -b_1 & b_2 & 0 &\dots & 0 \\
  a_2 & 0 & -b_2 & b_3 &\dots & 0 \\
  a_3 & 0 &  0 & -b_3 &\dots & \dots \\
  \dots & \dots &  \dots & \dots &\dots & \dots \\
  a_{N-1} & 0 &  0 & 0 & -b_{N-1} & b_N \\
  a_N & 0 &  0 & 0 & 0 & -b_N 
\end{pmatrix}, 
\end{align}
where $a_j\geq0$ and $b_j > 0$ for $j=1, \dots, N$ 
are non negative rates and we set $a := \sum_{j=1}^N a_j$.
The condition $b_j > 0$ is reasonable, since otherwise the loop is broken
somewhere.

The process $p(t)$ is generated by the equation $\dot p = \oAs p$.
We set $p=(u, v_1, \dots ,v_N)$ and understand this quantity as the
concentration of some particles.
We assume that for $t=0$ the total mass is concentrated in 
the first coordinate, i.e $p_0=(u_0,0, \dots, 0)$. 
The equation conserves positivity of $p$ and the whole mass
$u+v_1+...+ v_N = u_0$.
Thus, $p$ is a vector on the positive simplex in $\R^{N+1}$,
intersected by the hyperplane $u+v_1+...+ v_N = u_0$.
Of our interest is the first component, i.e. 
the amount of matter of particles of type $z_0$.

$\oAs$ is the generator of a special type of MPs.
It describe the change of types in the following way: Particles of type
$z_0$ can changes their type to type $z_i$ with rates $a_i$. The change of a 
particle of type $z_i$ back to type $z_0$ does not go in a direct way, but in $i$ steps.
Thus, we have an interaction between the $N+1$ types in $N$ loops
(see the picture).
\label{FMTME}

\begin{minipage}[b]{8cm}
\unitlength=2cm
\begin{picture}(15, 1.8)
\linethickness{0.2mm}
\put(3.25,1.25){\circle*{0.1}}         % 1 
\put(0.25,0.25){\circle*{0.1}}         % 2 
\put(1.75,0.25){\circle*{0.1}}         % 3 
\put(3.25,0.25){\circle*{0.1}}         % 4 
\put(6.0,0.25){\circle*{0.1}}         % N-1 
\put(7.5,0.25){\circle*{0.1}}         % N

\put(3.0,1.2){\vector(-3,-1){2.5}}  % 1 -> 2
\put(3.1,1.2){\vector(-3,-2){1.2}}     
% 1 -> 3
\put(3.25,1.15){\vector(0,-1){0.8}}     % 1 -> 4
\put(3.35,1.2){\vector(3,-1){2.5}}     % 1 -> N-1
\put(3.45,1.2){\vector(4.7,-1){3.8}}     % 1 -> N

\put(7.3,0.25){\vector(-1,0){1.0}}     % N-1 -> N
\put(5.8,0.25){\vector(-1,0){0.5}}     % N-1 -> N-2
\put(3.9,0.25){\vector(-1,0){0.5}}     % 5 -> 4
\put(3.0,0.25){\vector(-1,0){1.0}}     % 4 -> 3
\put(1.5,0.25){\vector(-1,0){1.0}}     % 3 -> 2
\put(0.4, 0.4){\vector(3,1){2.5}}  % 2 -> 1

\put(3.25,1.5){\makebox(0.0,0.0){$z_0$}}
\put(0.25,0.0){\makebox(0.0,0.0){$z_1$}}
\put(1.75,0.0){\makebox(0.0,0.0){$z_2$}}
\put(3.25,0.0){\makebox(0.0,0.0){$z_3$}}
\put(6.0,0.0){\makebox(0.0,0.0){$z_{N-1}$}}
\put(7.5,0.0){\makebox(0.0,0.0){$z_N$}}

\put(4.45,0.25){\makebox(0.0,0.0){$\dots\dots$}}  % between some dots
\put(0.9,0.7){\makebox(0.0,0.0){$b_1$}}   % 3 -> 1
\put(1.35,0.4){\makebox(0.0,0.0){$b_2$}}  % 3 -> 2
\put(2.85,0.4){\makebox(0.0,0.0){$b_3$}}  % 4 -> 3
\put(6.35,0.4){\makebox(0.0,0.0){$b_N$}}  % N -> N-1

\put(1.8,0.7){\makebox(0.0,0.0){$a_1$}} % 1 -> 2
\put(2.6,0.7){\makebox(0.0,0.0){$a_2$}} % 1 -> 3
\put(3.5,0.7){\makebox(0.0,0.0){$a_3$}} % 1 -> 4
\put(4.5,0.7){\makebox(0.0,0.0){$a_{N-1}$}} % 1 -> N-1
\put(6.2,0.7){\makebox(0.0,0.0){$a_N$}} % 1 -> N
\end{picture}
\end{minipage}

~

Easy calculations show that the stationary solution $\mu$ satisfying $A^* \mu = 0$ has the form
\begin{align*}
 \mu = \frac 1 Z\left(1, \frac{a_1+\dots+a_N}{b_1}, 
\frac{a_2+\dots+a_{N}}{b_2}, \frac{a_3+\dots+a_{N}}{b_3},\dots, 
\frac{a_N}{b_N}\right) u_0,  
\end{align*}
where $Z$ is the suitable normalization such that $\sum_{j=0}^{N}\mu_j=u_0$.
Obviously,
\begin{eqnarray}\label{e632}
Z = 1 + \sum_{i=1}^N { 1 \over b_i }  \sum_{j=i}^N a_j~.
\end{eqnarray}
For the zeroth coordinate we have
\begin{align*}
u(\infty) = { 1 \over Z} .
\end{align*}
Since any $b_j > 0$,  this stationary solution is unique
and it is the equilibrium state for any  initial condition.

Let us check, whether detailed balance with respect to $\mu$ is satisfied. 
We have to check, 
that $A_{ij}\mu_j = A_{ji}\mu_i$. Since $A_{1j}\mu_j = A_{j1}\mu_1 = 0$ 
for $j\geq2$, we obtain that $a_2 = a_3 = \dots a_N =0$. 
Hence, the evolution of the states $z_2, \dots, z_N$ is
not coupled to the evolution of $z_0$ and $z_1$. In this case, we get
$N=1$, the two dimensional case, where every MP has the
detailed-balance property. That means, apart from trivial situations, the MP under consideration does not have the
detailed balance property.

The equation $\dot p = \oAs p$ is equivalent 
to the following system for $p=(u,v_1, \dots, v_N)$
\begin{equation*}
\left\{
\begin{aligned}
 \dot u &=  -a u  + b_1v_1\\
 \dot v_1 &= a_1 u  - b_1u + b_2v_2\\
 \dot v_2 &= a_2u - b_2 v_2 + b_3v_3\\
 \dot v_3 &= a_3u - b_3 v_3 + b_4v_4\\
 \dots&\dots\\
 \dot v_{N-1} &=a_{N-1}u - b_{N-1} v_{N-1} + b_Nv_N\\
 \dot v_N &= a_Nu - b_N v_N.
\end{aligned}
\right.
\end{equation*}

Using the Laplace transform,  we get the following equation for $(u,v_1,\dots,v_N)$
\begin{equation*}
\left\{
\begin{aligned}
 (\lambda + a)u - u_0 &= b_1v_1\\
 (\lambda + b_1)v_1 &=a_1u + b_2v_2\\
 (\lambda + b_2)v_2 &=a_2u + b_3v_3\\
 (\lambda + b_3)v_3 &=a_3u + b_4v_4\\
 \dots&\dots\\
 (\lambda + b_{N-1})v_{N-1} &=a_{N-1}u + b_Nv_N\\
 (\lambda + b_1)v_N &=a_Nu.
\end{aligned}
\right.
\end{equation*}
This yields for $u$
\begin{eqnarray}\label{e664o} \nonumber
 (\lambda + a)u - u_0 = \left(\frac{a_1 b_1}{\lambda + b_1} +
   \frac{a_2b_1b_2}{(\lambda + b_1)(\lambda+b_2)} +
   \frac{a_3b_1b_2b_3}{(\lambda + b_1)(\lambda + b_2)(\lambda + b_3)} + \dots
 \right.\\ \label{e664}
 + \left.\frac{a_Nb_1b_2\cdots b_N}{(\lambda + b_1)(\lambda + b_2)\cdots(\lambda + b_N)} \right) u. 
\end{eqnarray}
We define the kernel
\begin{eqnarray*}
 k(\lambda) = \sum_{j=1}^N a_j k_j(\lambda),~~
k_j(\lambda) = \prod_{i=1}^j\frac{b_i}{\lambda + b_i}
\end{eqnarray*}
and hence the equation for the Laplace transformed variable $u$ reads
\begin{eqnarray}\label{e662}
 \lambda u - u_0 = -au + k(\lambda)u.
\end{eqnarray}

Now, we formulate the memory equation in terms of $t\geq0$
and some properties of the kernel.
For this purpose, we introduce some quantities,
connected with Lagrange polynomials (see
the appendix for details) with different support
points $b_1,...,b_N$.
Let 
\begin{eqnarray*}
\psi_i^j = \prod_{k=1, k\neq i}^j \frac{b_k}{b_k - b_i}~,
\end{eqnarray*}
assuming $b_i \not= b_k$ for $i \not= k $. From the theory of Lagrange polynomials it is well known that
\begin{align*}
k_j(\lambda) = \prod _{i=1}^j \frac{b_i}{\lambda +b_i} = \sum_{i=1}^j
  \frac{b_i}{\lambda+b_i}\psi_i^j. 
\end{align*}
Using this, we can transform $k_j(\lambda)$ back and obtain
\begin{eqnarray}\label{e645}
K(t) &=& \sum_{j=1}^N a_j K_j(t),\\ \label{e646}
K_j(t) &=&  \sum_{i=1}^j b_i \psi_i^j \e^{-b_i t}.
\end{eqnarray}

The assumption $b_i \not= b_j$ for $i \not= j $ is not
principal. If some or all $b_i$ coincide, all formulae of the following can
be obtained by some suitable limits. This is obviously done for the
Laplace transform $k(\lambda)$. For $K(t)$ we get more complicated terms,
involving not only exponential but also polynomials with degree, depending on
the frequency of the $b_i$. We do not bore the reader with this technical 
complexity, since this is well known in the theory of Lagrange polynomials.
Moreover, from a practical point of view, in a generic Markov
matrix all entries can be chosen differently. 

Surely, a different situation is, if the modeling requires equal $b_i$. This is the case for instance for DDEs.
The case is considered in detail in chapter 4.

Now, we are ready for the following
\begin{theo}\label{TheoremMPtoME}
 Let $p=(u,v_1, \dots,v_N)$ be the solution of 
$\dot p = \oAs p$ with $p_0=(u_0, 0, \dots, 0)$ where $\oAs$ is given via (\ref{MarkovGeneratorGeneralCase}). 
Then $t\mapsto u(t)$ solves the memory equation
 \begin{eqnarray}\label{e663}
  \dot u = -a u + \int_0^t K(t-s)u(s)\d s,~ u(0)=u_0,
 \end{eqnarray}
where
$K(t) = \sum_{j=1}^N a_j K_j(t)$ with 
$K_j(t) = \sum_{i=1}^j b_i \psi_i^j \e^{-b_i t}$ and $a=\sum_j a_j=k(0)$.
Moreover, $K(t)\geq 0$ and $u_\infty=1/Z$ with $Z$ given by (\ref{e632}).
\end{theo}
\begin{proof}
From the definition of $k(\lambda)$ it is clear that $u(\lambda)$ defined by
the MP is the solution to (\ref{e662}). If the inverse transformed function
$t\mapsto u(t)$  is regular enough, it is solution to (\ref{e663}).

Rewriting (\ref{e664}) as 
\begin{eqnarray}\label{e665}
\lambda u(\lambda) = 
{\lambda \over \lambda + a - \sum_{j=1}^N a_j k_j(\lambda)} ~ u_0
\end{eqnarray}
Since the $k_j(\lambda)$ are analytical functions and bounded on the right plane, so 
is $\lambda u(\lambda)$. Hence from the properties of the
Laplace 
transform it follows that $u(t)$ is continuously differentiable. Thus, it solves 
(\ref{e663}).

To calculate $u_\infty$ we use the representation (\ref{e665}) and investigate the
behavior of $k_j(\lambda)$ for $\lambda \rightarrow \infty$. We have
\begin{eqnarray*}
k_j(\lambda) &=& k_j(0) + \lambda k_j'(0) + o(\lambda)
\eyy
1 +  \lambda 
\left.\left(\frac{b_1b_2\cdots b_j}{(\lambda + b_1)(\lambda + b_2)
\cdots(\lambda + b_j)} \right)'\right|_{\lambda = 0 }
 + o(\lambda)
\eyy
1 - \lambda 
\left.\frac{b_1b_2\cdots b_j \cdot \left(b_1b_2\cdots b_j 
\sum_{i=1}^j {1 \over b_i} + o(\lambda) \right)
}{\big[(\lambda + b_1)(\lambda + b_2)
\cdots(\lambda + b_j)\big]^2} \right|_{\lambda = 0 }
 + o(\lambda)
\eyy
1 - \lambda \sum_{i=1}^j {1 \over b_i}  + o(\lambda).
\end{eqnarray*}
By definition $a=\sum_{j=1}^N a_j$, and hence, it follows from (\ref{e665}) 
\begin{eqnarray*}
u(\infty) &=&
\lim\limits_{\lambda \pfk \infty} \lambda u(\lambda) = 
\lim\limits_{\lambda \pfk \infty} 
{\lambda \over \lambda + a - \sum_{j=1}^N a_j 
\left[1 - \lambda \sum_{i=1}^j {1 \over b_i}  + o(\lambda)\right]
} ~ u_0
\eyy
{1\over 1+\sum_{j=1}^N a_j \sum_{i=1}^j {1 \over b_i} } ~ u_0
=
{1\over 1 + \sum_{j=1}^N { 1 \over b_j }  \sum_{i=j}^N a_j } ~ u_0~,
\end{eqnarray*}
what is exactly the zeroth coordinate of $\mu$, i.e. $u_\infty=1/Z$.

The positivity of the $K_j(t)$, $t \geq 0$ follows from their
representation with simplex integrals (see the appendix). We have
\begin{eqnarray*}
K_j(t) &=& \sum_{i=1}^j b_i \psi_i^j \e^{-b_i t}
=
\left. 
\int_{S_j} (-1)^{j-1} f^{(j-1)}\big( \la \alpha , s \ra t \big) 
\right|_{s_j = 1 - s_1 - s_2 - s_{j-1}} \d s_{j-1} \cdots  \d s_1
\end{eqnarray*}
with  $f(x) = \e^{-x t}$ and
$\la \alpha , s \ra = \alpha_1 s_1 +  
\alpha_2 s_2 + \ldots + \alpha_j s_j$. Since
$(-1)^{j-1} f^{({j-1})}\big( \la \alpha , s \ra t \big) = 
t^{j-1} \e^{-  \la \alpha , s \ra  t} \geq 0$ and any $a_j\geq0$, we conclude the positivity of $K_j(t)$ and therefore also $K(t) \geq 0$.
This completes the proof of the theorem.
\end{proof}

\subsection{From Memory to Markov}

We consider memory equations of the form
\begin{align*}
 \dot u(t) = -a u + K\ast u = -a u + \int_0^t K(t-s)u(s)\d s,
\end{align*}
where $a>0$ is a real parameter and $K$ is a positive kernel. 
The aim is to embed the evolution of $u$ into a MP 
introducing new variables.

Our main assumptions are
$K(t)\geq0$ and $\int_0^\infty K(t)\d t = a$.
Clearly, starting with some given $K(t)$ we want to end up with a kernel 
of the shape (\ref{e645}-\ref{e646}). Then going forward 
to a kernel like in (\ref{e664}),
the entries of the Markov generator matrix can be taken immediately.

The kernels (\ref{e646}) are positive although this are linear combinations
of exponential with -- maybe -- negative coefficients.

It may seem that any nonnegative kernel $K(t)$ can be presented in such a
form. But this is not the case. We show this in a 

~

{\textbf  {Counterexample:}} Let
\begin{eqnarray*}
K(t) =  3 \e^{- t} - 8 \e^{- 2 t} + 6 \e^{-3 t}
\end{eqnarray*}
and 
\begin{eqnarray*}
f(t) =  \e^{4 t} K(t) =  3 \e^{3 t} - 8 \e^{2 t} + 6 \e^{ t}
\end{eqnarray*}
$f(t)$ has a unique minimum $f(0.215315...) = 0.8590718...$. 
Thus $K(t) \geq 0$. 

Seeking for coefficients $A,B,C,D,E,F,G$ (this is the representation 
(\ref{e646})) with
\begin{eqnarray*}
K(t) &=& A  \e^{-3 t} + B \e^{-2 t} + C \e^{- t} + 
D { \e^{-t} - \e^{-2 t} \over 1 } + E { \e^{-t} - \e^{-3 t} \over 2 } +
F { \e^{-2 t} - \e^{-3 t} \over 1 } + \\ &+& G 
\left(  {\e^{-t} \over 1 \cdot 2 } +  {\e^{-2t} \over (-1) \cdot 1 } 
+  {\e^{-3t} \over 1 \cdot 2 } \right)
\end{eqnarray*}
the resulting system for the coefficients leads to
\begin{eqnarray*}
0 = 2 + D + E + F + B + C
\end{eqnarray*}
that does not have nonnegative solutions.

~

We think, there is no hope to find a corresponding MP for an arbitrary
nonnegative kernel.
Therefore we go another way and try to derive a class
of sensible kernels starting from physical
considerations. Furthermore, the following reasoning shows how
the time interval of the memory effect is
connected with rates of the loops of the MP.

First of all we have to ask: How one can model a meaningful kernel for a ME. 
We can assume that the dependence on the past is concentrated at some time
point before the present, say $t-t_1$, ... $t-t_N$ where $t_j$ are ordered time
values, i.e. $0<t_1<t_2<\dots <t_N$, with some coefficients
$\gamma_1, ..., \gamma_N$ with $\gamma_i \geq 0$ and $\sum \gamma_i = 1$ that
gives the relative proportion of each time point. The corresponding  memory
kernel of such an ansatz is
\begin{eqnarray*}
\tilde K(t) = \sum_{j=1}^N \gamma_j \delta(t-t_j)
\end{eqnarray*} 
(here $\delta$ means the ``$\delta$-function'', the ``density'' of the Dirac
measure). The kernel $\tilde K$ occurs when starting from a discrete 
time model, like equation
(\ref{e172}). Clearly, it is a first guess. A real memory kernel seems to be more smeared. Therefore, we can try to find kernels $\tilde K_j(t)$ with
mean time at $t_j$, i.e 
\begin{align*}
 \int_0^\infty \tilde K_j(t) \d t = a,~~~
 \int_0^\infty t \tilde K_j(t) \d t = \int_0^\infty t \delta(t-t_j) \d t = t_j.
\end{align*}
We will show that such kernels $ \tilde K_j(t)$ can be found and
it is possible to find a suitable MP for them. Note, that
this does not determine the kernels $\tilde K_j$ uniquely, of course.

We show that our kernels  of shape  (\ref{e645}) are suitable for this.

\begin{prop}
 Let a sequence $0<t_1<t_2<\dots <t_N<\infty$ be given where $(t_{j}-t_{j-1})$ are pairwise distinct. There are kernels $K(t) = \sum_{j=1}^N a_j K_j(t)$ such that $K\geq 0$ and $\int_0^\infty K(t)\d t = a$ and  $\int_0^\infty tK_j(t)\d t = t_j$. 
\end{prop}
\begin{proof}
We define $b_j\in\R$ via $t_i=\sum_{j=1}^i \frac{1}{b_j}$. Since the $t_i$ are
ordered, we get $b_j>0$. Since $(t_{j}-t_{j-1})$, the $b_j$ are pairwise distinct. We define 
\begin{align*}
 K(t) = \sum_{j=1}^N a_j K_j(t), ~~ \mathrm{where~~} K_j(t) = \sum_{i=1}^j b_i
 \psi_i^j \e^{-b_i t}. 
\end{align*}
We prove that $K$ satisfies the desired properties.
Using the Laplace transform, we get
\begin{align*}
 \mathcal{L}(K_j(t))(\lambda) = \sum_{i=1}^j b_i \psi_i^j
 \frac{1}{\lambda+b_i} = \prod _{i=1}^j \frac{b_i}{\lambda + b_i}
 =:k_j(\lambda). 
\end{align*}
This yields $\int_0^\infty K_j(t)\d t = k_j(\lambda = 0) = 1$. Moreover,
$\int_0^\infty t K_j(t)\d t = -k_j'(\lambda = 0)$. We have 
\begin{align*}
 k_j'(\lambda) =
 \sum_{i=1}^j\frac{b_1}{\lambda+b_1}\cdot\frac{b_2}{\lambda+b_2}\cdots\frac{-b_i}{(\lambda+b_i)^2}\cdots\frac{b_{j-1}}{\lambda+b_{j-1}}\cdot\frac{b_j}{\lambda+b_j}  
\end{align*}
This yields $-k_j'(\lambda=0) = \sum_{i=1}^j \frac{1}{b_i} = t_j$,
i.e. $\int_0^\infty t K_j(t)\d t =  t_j$. 
\end{proof}

\begin{theo}%[(ME)$\Rightarrow$(MP)]
 Let $K(t)$ be a memory kernel of the form
\begin{align*}
 K(t) = \sum_{j=1}^N \alpha_j K_j(t), ~~ 
\mathrm{where~~} K_j(t) = \sum_{i=1}^j b_i \psi_i^j \e^{-b_i t}.
\end{align*}
and $\alpha = \sum_j \alpha_j$. Let $u$ be the solution to
the equation $\dot u(t) = -\alpha u + K\ast u$ with $u(0)=u_0$. 
Then, there is a MP $\dot p = \oAs p$
in $\R^{N+1}$ generated by a Markov matrix $\oA$ and an
initial condition $p(0)$ such that $u(t) = p_0(t)$.
\end{theo}
\begin{proof}
Define the Markov generator matrix via $a=\alpha$, 
$a_i=\alpha_i$, $b_i=\beta_i$. The initial condition for the MP is $p_0=(u_0,0,\dots,0)$.
The claim follows.
\end{proof}

For the asymptotic behavior of the ME, we immediately get the following statement.
\begin{cor}
Let $K(t)=\sum_{j=1}^N a_i K_i(t)$, where $K_i(t) = \sum_{j=1}^i b_j \psi_j \e^{-b_j t}$ and $a = \sum_j a_j$. Let $u$ be the solution to
the equation $\dot u(t) = -a u + K\ast u$ with $u(0)=u_0$. Then $u(t)\rightarrow u_\infty$ as $t\rightarrow \infty$, where $u_\infty = \frac{1}{Z}u_0$ and $Z$ is given by (\ref{e632}).
\end{cor}

\subsection{Remarks}

\begin{enumerate}

%\subsubsection{More general kernels}

\item
Kernels like 
$k_j(\lambda) = \prod_{i=1}^j\left(\frac{b_i}{\lambda+b_i}\right)^{m_i}$ 
with suitable chosen $m_i\in\N$ may
approximates a $\delta$-kernel better. Especially it allows to take
into account
more moments then only the first one, or equivalently to allow
the $b_i$ to be equal. This is possible without any principal
problems (see the note above Theorem 3.1).
A special case is treated in the next chapter, 
where one delay is approximated arbitrary precise.
To prove positivity of the corresponding functions Lemma 
\ref{reslem} from the appendix can be used.

Kernels like in  (\ref{e664}) are rational functions of degree $N$,
having poles on the left plane. They
approximate meromorphic functions. This makes one able to consider more
general kernels then linear combinations of exponents -- 
at least approximately.

%\subsubsection{Other Markov processes}

\item
There are other (similar) MP that lead to a ME and vice versa. 
For example the MP with the generator
\begin{align*}
\oAs=\begin{pmatrix}
 -a & c_1 & c_2 & c_3 & \dots & c_N\\
  a & -c_1 - b_1 & 0 & 0 &\dots & 0 \\
  0 & b_1 & -c_2 - b_2 & 0 &\dots & 0 \\
  0 & 0 &  b_2 & -c_3 - b_3 &\dots & 0 \\
  \dots & \dots &  \dots & \dots &\dots & \dots \\
  0 & 0 &  0 & 0 & b_{N-1} & -c_N 
\end{pmatrix}, 
\end{align*}
can also be used for embedding the presented exponential kernels. 
Such MP can be understood in the same manner like at the picture on page
\pageref{FMTME} but with reversed arrows.
Although this approach is more difficulty from a technical 
point of view.

%\subsubsection{Analytical applications}

\item
The presented results can be applied in various manner. 
We focus on ordinary differential equations to present the general 
idea. Linear MEs in infinite dimensional space
like diffusion equations with time depending diffusion coefficients 
are also possible.

Moreover, the well known tools for investigating MP, 
like inequalities for Lyapunov functions (see \cite{stephan2})
can now be carried over to explore ME.

\end{enumerate}

%\newpage
%%%%%%%%%%%%%%%%%%%%%%%%%%%%%%%%%%%%%%%%%%%%%%%%%%%%%%%%%%%%%%%%%%%%%%%%%%%%%%%%%%%%%%%%%%%%%%% 

\section{Special Markov process leads to a Delay Differential equation}\label{DDE}
In this section we consider a special form of the MP. We define $a_j=0$ for $j={1,2, \dots, N-1}$ and put $a_N=a$ and $b_j=b\in\R$. Using the observation from the last section we consider a general cyclic
MP with one single but long loop. The MP in $\R^{N+1}$ is generated by the matrix
\begin{align*}
\oAs=\begin{pmatrix}
 -a & b & 0 & \cdots & 0\\
  0 & -b & b & \cdots& 0 \\
  0 & 0 & -b &\cdots& 0 \\
  \vdots &  & & \ddots & b \\
  a &  \cdots & 0 &\cdots & -b
\end{pmatrix}.
\end{align*}
We assume the
initial mass is concentrated in the first reservoir. Then, the equation reads $\dot{p}(t) = \oAs p(t)$ with $p(0) = p_0$, 
where $p = (u, v_1, v_2, \dots, v_n)^T$ and $p_0 = (u_0, 0, \dots,0)^T$.

\begin{minipage}[b]{10cm}
\unitlength=2cm
\begin{picture}(15, 1.8)
\linethickness{0.2mm}
\put(3.25,1.25){\circle*{0.1}}         % 1 
\put(0.25,0.25){\circle*{0.1}}         % 2 
\put(1.75,0.25){\circle*{0.1}}         % 3 
\put(3.25,0.25){\circle*{0.1}}         % 4 
\put(6.0,0.25){\circle*{0.1}}         % N-1 
\put(7.5,0.25){\circle*{0.1}}         % N

\put(3.45,1.2){\vector(4.7,-1){3.8}}     % 1 -> N

\put(7.3,0.25){\vector(-1,0){1.0}}     % N-1 -> N
\put(5.8,0.25){\vector(-1,0){0.5}}     % N-1 -> N-2
\put(3.9,0.25){\vector(-1,0){0.5}}     % 5 -> 4
\put(3.0,0.25){\vector(-1,0){1.0}}     % 4 -> 3
\put(1.5,0.25){\vector(-1,0){1.0}}     % 3 -> 2
\put(0.4, 0.4){\vector(3, 0.9){2.7}}  % 2 -> 1

\put(3.25,1.5){\makebox(0.0,0.0){$z_0$}}
\put(0.25,0.0){\makebox(0.0,0.0){$z_1$}}
\put(1.75,0.0){\makebox(0.0,0.0){$z_2$}}
\put(3.25,0.0){\makebox(0.0,0.0){$z_3$}}
\put(6.0,0.0){\makebox(0.0,0.0){$z_{N-1}$}}
\put(7.5,0.0){\makebox(0.0,0.0){$z_N$}}

\put(4.45,0.25){\makebox(0.0,0.0){$\dots\dots$}}  % between some dots
\put(0.9,0.7){\makebox(0.0,0.0){$b$}}   % 3 -> 1
\put(1.35,0.4){\makebox(0.0,0.0){$b$}}  % 3 -> 2
\put(2.85,0.4){\makebox(0.0,0.0){$b$}}  % 4 -> 3
\put(6.35,0.4){\makebox(0.0,0.0){$b$}}  % N -> N-1

\put(6.2,0.7){\makebox(0.0,0.0){$a$}} % 1 -> N
\end{picture}
\end{minipage}

The stationary solution is
\begin{eqnarray*}
\mu = \frac{1}{Z}\left( {1 \over a} , {1 \over b} , {1 \over b} , ..., {1 \over b}
\right)^Tu_0\in\R^{N+1},
\end{eqnarray*}
where $Z = \frac{1}{a} + \frac{N}{b} = {b + a N \over a b}$. 
Note, the system does not have the detailed balance property.

We get
\begin{align*}
 (\lambda + a) \hat u - u_0 = a \left(\frac{b}{\lambda + b}
 \right)^Nu. 
\end{align*}
It holds
\begin{align*}
  \left(\frac{b}{\lambda + b}\right)^N = \mathcal {L}\left(\frac{b^N}{(N-1)!}t^{N-1}\e^{-b t}\right)(\lambda).
\end{align*}

\begin{minipage}[b]{7cm}
Hence, we get
\begin{align*}
 \dot u(t) &= -a u(t) + \frac{ab^N}{(N-1)!}  \int _0^t
 s^{N-1}\e^{-b s}u (t-s) \D s \\
 &= -a \left( u(t) - \int _0^t K_N(s) u (t-s) \D s\right),
\end{align*}
where we introduced the kernel 
$$K_N(t) := \frac{b^N}{(N-1)!}  t^{N-1}\e^{-b t}.$$
\end{minipage}
\hfill
\begin{minipage}[b]{7cm}
\unitlength=1cm
\begin{picture}(8,5)
\linethickness{0.2mm}
\put(0.0,0.0){\includegraphics[width=7cm]{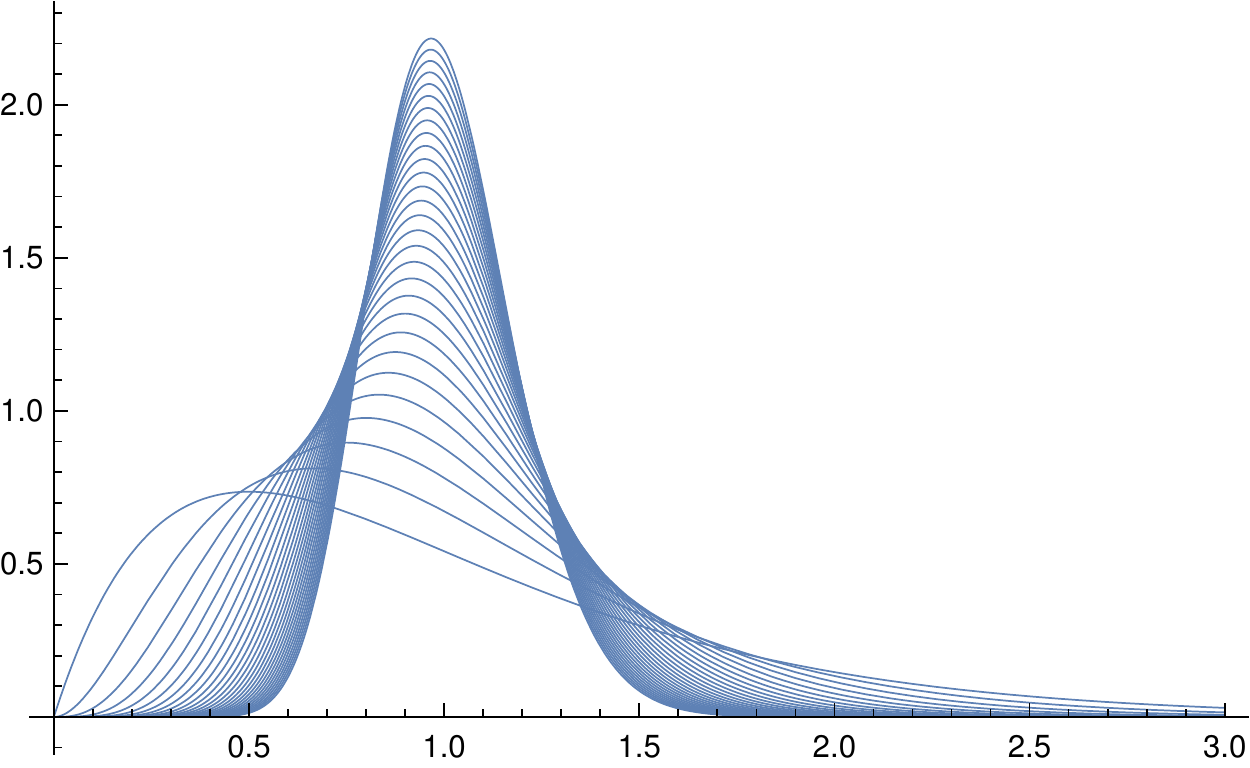}}
\put(3.5,4.5){Kernel $K_N(t)$ for:}
\put(3.5,4.0){$b=\frac{N}{T}$}
\put(3.5,3.5){$T=1$}
\put(3.5,3.0){$N=2,\dots, 30$}
%\put(1.0,3.6){\makebox(0.0,0.0){$K(t)$}}
\put(6.8, 0.6){\makebox(0.0,0.0){$t$}}
\end{picture}
\end{minipage}

~

A delay equation can be understood as a memory equation with a $\delta$-kernel. To do this, we fix $T>0$ and introduce $\delta_T(t)=\delta(t-T)$. 
We get
\begin{align*}
 \int_0^\infty \delta_T(t)\e^{-\lambda t} \D t = \int_0^\infty \delta(t - T)\e^{-\lambda t} \D t = \e^{-\lambda T}.
\end{align*}
Moreover, for $t>T$ we have
\begin{align*}
 u(t-T) &= \int _0^\infty u(s) \delta(t-T-s)\D s = \int _0^\infty u(s) \delta_T(t-s)\D s =\\
  &= \int _0^t u(s) \delta_T(t-s)\D s = u(t)\ast \delta_T(t).
\end{align*}
Hence,
\begin{align*}
 \mathcal L(u(t-T))(\lambda) &=  \hat u(\lambda) \e^{-\lambda T}.
\end{align*}

Putting $b=\frac{N}{T}$, we approximate the Laplace transform of the kernel $\delta_T$, i.e.
\begin{align*}
  \mathcal L(\delta_T)(\lambda) = e^{-\lambda T} \approx 
\left(1+\frac{\lambda T}{N}\right)^{-N} 
= \left(\frac{\frac N T}{\frac N T + \lambda}\right)^N 
= \mathcal{L}(K_N(t))(\lambda).
\end{align*}
Hence, we conclude
\begin{align*}
  \mathcal{L}(K_n(t))(\lambda) \xrightarrow{n\rightarrow \infty} e^{-\lambda T} = \mathcal L (\delta(t-T))(\lambda),
\end{align*}
and the limiting (DDE) reads as
\begin{align*}
 \dot u = \begin{cases}
           -a u(t), ~~&\mathrm{if~~} 0\leq t\leq T\\
           -a u(t) + au(t-T), ~~&\mathrm{if~~} t\geq T,
          \end{cases}
\end{align*}
or equivalently
\begin{align}\label{LimittingDDE}
 \dot u = -a u(t) + au(t-T), \mathrm{~~for} ~~t\geq T, \mathrm{~~~and~~~} u|_{[0,T]}(t) = \e^{-a t}u_0. 
\end{align}
Let us note that the initial condition $u|_{[0,T]}(t) = \e^{-a t}u_0$ results 
from the modeling ansatz. No other initial condition is possible.

Let us compute the limiting stationary solution for $N\rightarrow \infty$ 
of the first coordinate of the MP. This means the MP has long loops, 
but mass is transferred with a high rate. We have $Z = \frac{Na + b}{ab}$. 
Putting $b=\frac{N}{T}$, we conclude for the zeroth coordinate of the 
stationary solution
\begin{align*}
\mu_0 = \frac{1}{Za} = \frac{b}{Na+b} = \frac{\frac N T}{Na + \frac N T} = \frac{1}{1+aT}.
\end{align*}
\begin{minipage}[b]{5cm}
 \bigskip\bigskip
 The solution of the DDE and the stationary solution $\mu_0$ of the MP can be seen in the picture. The solution of the DDE converges nicely to $\mu_0$.
\end{minipage}
~~~~
\begin{minipage}[l]{10cm}
\unitlength=1cm
\begin{picture}(8,3)
\linethickness{0.2mm}
\put(0.0,0.0){\includegraphics[width=7cm]{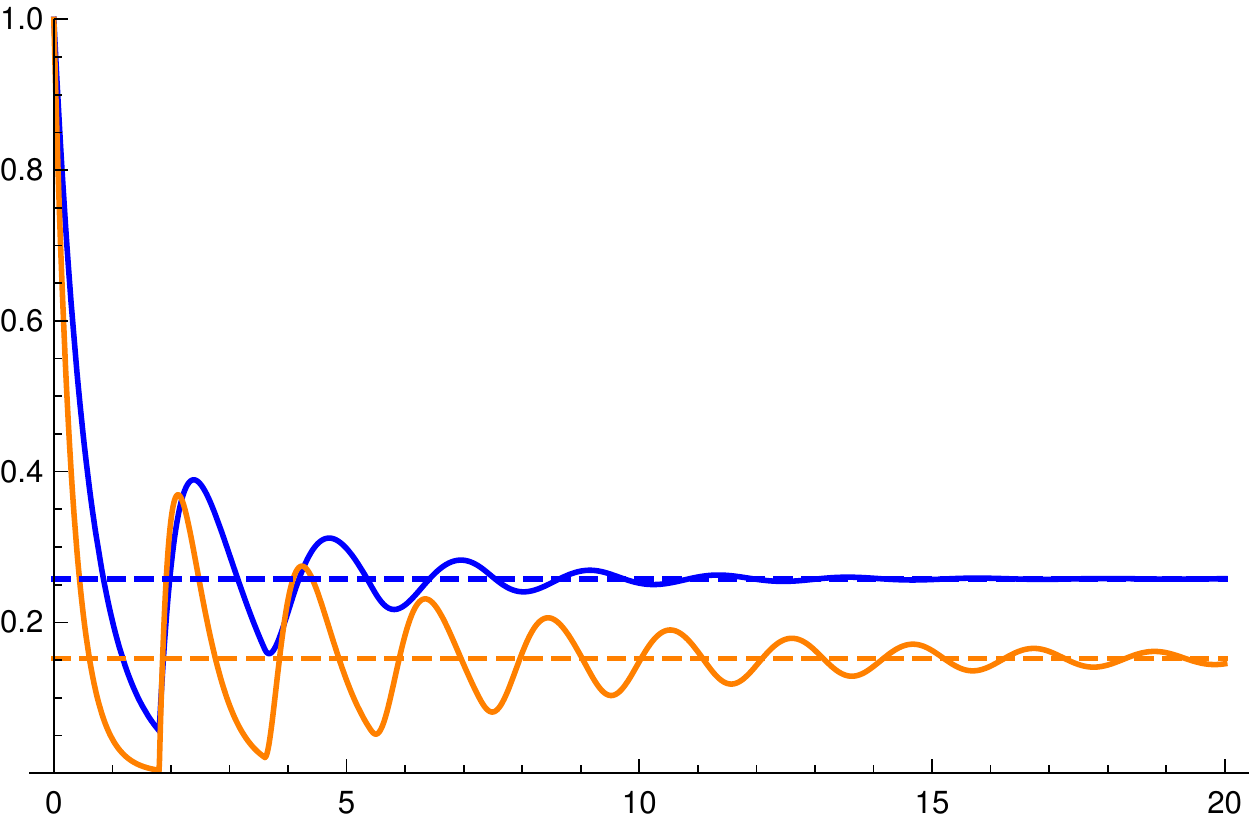}}
\put(1.0,3.5){Solution for DDE (\ref{LimittingDDE}) and equilibrium $\frac{1}{1+aT}$} 
\put(1.0,3.0){for $u_0=1$ with parameters:}
\put(1.0,3.5){}
\put(4.0,2.3){$T=1.8$}
\put(7.0,1.5){$a=1.6$}
\put(7.0,0.5){$a=3.1$}
\end{picture}
\end{minipage}

\begin{minipage}[b]{8cm}
\end{minipage}
\hfill

Finally, we remark some properties of the spectrum.
The spectrum of the DDE is given by inserting $\e^{\lambda t}$ for $\lambda\in\mathbb{C}$ into the equation (see e.g. \cite{Smith}). This yields for given $a,T\geq0$ the equation
\begin{align}\label{CharEquationDDE}
 \lambda  = - a + a \e^{-\lambda T}. 
\end{align}
This transcendental equation (in $\lambda\in\mathbb{C}$) has in general an infinite discrete amount of solutions.

The eigenvalues of $\oAs$ for fixed $N\in\N$ are given by the characteristic equation
\begin{align*}
\phi(\lambda) = -a b^{N-1} + (\lambda+b)^{N-1} (\lambda+a)=0,
\end{align*}
that can be computed easily.
Hence, setting $b=\frac{N}{T}$ we get $\phi(\lambda)=0$ if and only if
\begin{align*}
 \frac{a}{a+\lambda} = \left(\frac{\lambda+b}{b}\right)^{N-1} = \left(1 + \frac{\lambda T}{N}\right)^{N-1}.
\end{align*}
For $N\rightarrow \infty$, right hand side converges to $\e^{\lambda T}$. So, in the limit $\lambda\in\mathbb{C}$ satisfies the equation \begin{align*}
\frac{a}{a+\lambda} = \e^{\lambda T},                                                                                               \end{align*}
i.e. the same equation as (\ref{CharEquationDDE}). Hence, one can say that not only the solution converges but also the spectrum of the MP and of the ME converges to each other. Note, that the convergence of the spectrum is very slow, as the convergence of the exponential function is.

\section{Appendix}

\subsection{Laplace transform}

Here, we summarize some facts of the Laplace transform. More details can be found, e.g. in 
\cite{LaplaceRef1}. 
For a given function $u: [0,\infty)\in t\mapsto u(t)\in\R$ that does not grow faster than an exponential function in time, the Laplace transform is defined by
\begin{eqnarray*}
\hat u (\lambda) = (\mathcal L u)(\lambda) = 
\int_0^\infty e^{-\lambda t} u(t) \d t.
\end{eqnarray*}

We use the following formulas that can be checked easily:
\begin{eqnarray*}
\mathcal L( \dot{u} )(\lambda) &=& \lambda \hat u (\lambda) - u_0 \\
\mathcal L(K \ast u) &=& (\mathcal L K) \cdot (\mathcal L u) \\
\mathcal L(\e^{-a \cdot})(\lambda) &=& { 1 \over \lambda + a} \\
\mathcal {L}\left(\frac{1}{(n-1)!}t^{n-1}\e^{-a t}\right)(\lambda) &=&
\frac{1}{(\lambda+a)^n}.
\end{eqnarray*}
The Laplace transform has an interesting asymptotic behavior.
The limit for large times $u(t)  \stackrel{t \pfk \infty}{\pf}  u_\infty$
can be calculated with the Laplace transform. It holds
$\lambda \hat u (\lambda) \stackrel{\lambda \pfk 0}{\pf}  u_\infty$.
Thus, there is no need to know the whole solution $u(t)$ if one is 
interested only in the equilibrium case.
This is important, since, in general for non-autonomous equations, 
the equilibrium case cannot be calculated by setting $\dot{u}=0$.

Let us note that the uniform convergence on compact sets of $t \in \R_+$
carries over to  uniform convergence on compact sets of $\lambda$ in the
domain of analyticity.

To carry over positivity properties between the original and the
transformation the following lemma is useful:
\begin{lem}\label{reslem}
 Let $K(t)=\sum_{j=1}^N\gamma_j \e^{-\alpha_j t}$ with its Laplace transform $k(\lambda) = \sum_{j=1}^N\gamma_j \frac{1}{\lambda+\alpha_j}$.
 Then $K(t)\geq0$ if and only if $\sum_{j=1}^N \frac{\gamma_j}{(\lambda+\alpha_j)^m} \geq 0$ for any $m\in\N$.
\end{lem}
\begin{proof}
 Let $K(t)\geq0$. Since $K(0)\geq0$, we get $\sum_{j=1}^N\gamma_j\geq0$, i.e. the claim holds for $m=0$.
 For $m\geq0$, we get $0\leq \int_0^\infty t^m K(t)\e^{-\lambda t}\d t =
 (-1)^m k^{(m)}(\lambda) = \sum_{j=1}^N\frac{\gamma_j}{(\lambda +
   \alpha_j)^{m+1}}$ what proves the claim in one direction. 
 
 For the other direction, we put $\lambda = \frac n t$ and $m+1 = n$. Then
 \begin{align*}
  0\leq \sum_{j=1}^N\frac{\gamma_j(\frac{n}{t})^n}{(\frac{n}{t} + \alpha_j)^n} = \sum_{j=1}^N\frac{\gamma_j}{(1 +
  \frac{\alpha_j n}{t})^n} = \sum_{j=1}^N \gamma_j \left(1 +
  \frac{\alpha_j t}{n}\right)^{-n} \rightarrow \sum_{j=1}^N \gamma_j \e^{-\alpha_j t}, \mathrm{~~as~~}n\rightarrow \infty,
 \end{align*}
 which proves the claim of the lemma.
\end{proof}

%\newpage

\subsection{Simplex integrals}

In Theorem \ref{TheoremMPtoME}, we proved the positivity of the 
kernel $K(t)$ using an integral over a simplex. This is based on the
following observation.

Let $S_{n-1} \subset \R^{n}$ be the simplex, defined as 
\begin{eqnarray*}
S_{n-1} = \{ s \in \R^{n}~|~ s_i \geq 0,~s_1+...+s_n = 1\}.
\end{eqnarray*}
We consider functions $g:\R^n \pf \R$ and their integrals over 
$S_{n-1}$. We have
\begin{align*}
\int\limits_{\rS_{n-1}} &g(s) \d\sigma(s) = 
\frac{1}{\sqrt{n}}
\int\limits_{S_{n-1}} g(s_1,s_2,...,s_{n-1},1{-}s_1{-}\ldots{-}s_{n-1}) 
\d{}s_1 \cdots \d{}s_{n-1}=\\
&=
(n-1)!\int\limits_0^1  \d{}s_1 \!\!
\int\limits_0^{1{-}s_1}\!\! \d{}s_{2} \!\!\!\!\!\!
\int\limits_0^{1{-}s_1{-}s_{2} }\!\!\!\!\!\! \d{}s_{3}~~ \cdots 
\int\limits_0^{1{-}s_n{-} \ldots {-}s_{n-2}} \!\!\!\!\!\!\!\!\!\d{}s_{n-1} ~
g(s_1,s_2,...,s_n)\Big|_{s_{n}=1{-}s_1 {-} \ldots {-}s_{n-1}}~,
\end{align*}
where $\sigma(\d{}s)$ is the Lebesgue measure on $S_{n-1}$ and $\sqrt{n}$
is the volume of $S_{n-1}$. 

Let $f:\R \pf \R$ be a smooth enough function, $f^{(k)}$ its $k$-
derivative and $x_1,...,x_n$  be given different real values.
Set $g(s) = f(\la x, s\ra)$,
where  $\la x, s\ra = x_1s_1+ x_2s_2 + \ldots + x_n s_n $ is the scalar product
in $\R^n$.

Now, using induction one can prove that
\begin{eqnarray*}
\sum_{i=1}^n f(x_i) \prod_{j\not=i}^n {1 \over x_i - x_j }
=
\int\limits_{S_{n-1}} f^{(n-1)}(\la x, s\ra)
\sigma(\d{}s)~.
\end{eqnarray*}

This formula gives a powerful tool to switch between expressions
connected with Lagrange polynomials and expressions
connected with
simplex integrals. In Theorem \ref{TheoremMPtoME}, we used this 
formula with $f(x) = \e^{- x t}$.

\subsection{Lagrange polynomials}

Here we summarize basic facts from the theory of Lagrange polynomials.
Let 
\begin{eqnarray*}
L_i^j (x) = \prod_{k=1, k\neq i}^j \frac{x - x_k }{x_i - x_k}~,
\end{eqnarray*}
assuming $x_i \not= x_k$ for $i \not= k $.
Obviously $L_i^j (x)$  is a polynomial of degree $j-1$ and
we have $L_i^j (x_k) = \delta_{ik}$ with $\delta_{ik}$
the Kronecker symbol. Hence, the polynomial
\begin{eqnarray*}
P(x) = \sum_{i=1}^j p_i L_i^j (x) 
\end{eqnarray*}
 of degree $j-1$ 
satisfy $P(x_i)=p_i$.

Now, let us fix $z\in\R$. 
Seeking for a polynomial $P(x)=q_0+q_1 x + ... + q_{j-1}x^{j-1}$ with the
condition $P(x_i) = p_i = {x_i \over z + x_i}$, 
we get coefficients
$q_i$ with $q_0 = \prod_{i=1}^j \frac{x_i}{z +x_i}$ among them. Hence,
we have on the one hand
\begin{eqnarray*}
P(0) = q_0 =  \prod_{i=1}^j \frac{x_i}{z +x_i}
\end{eqnarray*}
and on the other hand
\begin{eqnarray*}
P(0) = \sum_{i=1}^j p_i L_i^j (0) = 
\sum_{i=1}^j {x_i \over z + x_i}  
\prod_{k=1, k\neq i}^j \frac{(- x_k) }{x_i - x_k}
=
\sum_{i=1}^j {x_i \over z + x_i} 
\prod_{k=1, k\neq i}^j \frac{x_k}{x_k - x_i}~.
\end{eqnarray*}
It follows
\begin{eqnarray*}
\prod_{i=1}^j \frac{x_i}{z +x_i}
=
\sum_{i=1}^j {x_i \over z + x_i} 
\prod_{k=1, k\neq i}^j \frac{x_k}{x_k - x_i}~.
\end{eqnarray*}
Note, in our explanation we use $\psi_i^j = (-1)^{j-1} L_i^j(0)$ and put $z=\lambda$.

%-------------------------------------------------------------------------
\end{document}